\documentclass{wsbart/wsbart}
\usepackage{lmodern}
\usepackage{amsmath,amssymb,amsthm,mathtools}
\usepackage{booktabs,microtype,dsfont}
\usepackage{xurl}

\usepackage[maxnames=99,style=trad-plain,backend=biber,isbn=false,
uniquename=init]
{biblatex}

\DeclareRedundantLanguages{English}{english}

\renewcommand{\mathbb}{\mathds}
\renewcommand{\leq}{\leqslant}
\renewcommand{\geq}{\geqslant}

\newcommand{\BV}{\mathrm{BV}}

\newtheorem{theorem}{Theorem}
\newtheorem{proposition}[theorem]{Proposition}
\newtheorem{lemma}[theorem]{Lemma}
\theoremstyle{remark}
\newtheorem{remark}[theorem]{Remark}

\newcommand{\R}{\mathbb R}
\DeclareMathOperator{\E}{\mathbb E}
\DeclareMathOperator{\PP}{\mathbb P}
\newcommand{\dd}{\,\mathrm d}
\newcommand{\Id}{\mathrm{Id}}
\newcommand{\Lip}{\operatorname{Lip}}

\newcommand{\supp}{\operatorname{supp}}

\newcommand{\one}{\mathbf 1}
\newcommand{\cP}{\mathcal P}
\newcommand{\cS}{\mathsf S}
\newcommand{\cG}{\mathcal G}
\newcommand{\ip}[2]{\langle #1,#2\rangle}
\newcommand{\norm}[1]{\lVert #1\rVert}

\title{Quantitative stability of optimal transport\\via regularization}
\author[1]{Songbo Wang}
\author[2]{Xiaozhen Wang}
\affil[1]{CEREMADE, Universit\'e Paris-Dauphine, Universit\'e PSL, Paris, France}
\affil[2]{CMAP, \'Ecole polytechnique, Institut polytechnique de Paris, Palaiseau, France}
\keywords{optimal transport, quantitative stability, regularization, convex duality, bounded variation}

\begin{document}
\maketitle
\begin{abstract}
We prove quantitative stability estimates for quadratic optimal transport
maps by regularizing a Brenier potential and estimating the error in the
duality gap. The source has an $L^p$ density, $p>1$, relative to a
probability measure with a BV density. When one target lies in a fixed
ball, we obtain the $W_2$ H\"older exponent $(p-1)/(4p-2)$. The constant
depends on the dimension through three explicit source quantities: the
variance, the total variation of the reference density and the $L^p$
norm of the weight. The proof uses a second-order bound on the averaged
regularized Fenchel gap. By restricting the slopes of the potential,
we also treat targets with a bounded moment of any order greater than
two. In both estimates, the condition is imposed on one target, while
the other can be any probability measure with finite second moment.
For Gaussian and isotropic log-concave sources, the constants grow
polynomially with the dimension. For $L^p$ changes of Gaussian measure,
the factor in the bounded-target estimate is $d^{(p-1)/(4p-2)}$.
For each $p$, a fixed source in two dimensions proves optimality of both
stability exponents in this source class. A Gaussian construction
provides polynomial lower bounds on the dimension dependence.
\end{abstract}

\section{Introduction and main results}\label{sec:intro}

Let $\rho\in\cP_2(\R^d)$ be absolutely continuous. For
$\mu\in\cP_2(\R^d)$, denote by $T_\mu=\nabla\phi_\mu$ the Brenier map from
$\rho$ to $\mu$, for the quadratic cost. Quantitative stability asks how
$\norm{T_\mu-T_\nu}_{L^2(\rho)}$ depends on a Wasserstein distance between
$\mu$ and $\nu$. We allow discrete targets and impose a support or
higher-moment bound on one target. All norms are Euclidean, and
$W_r$ denotes the usual $r$-Wasserstein distance.

Stability under regularity assumptions goes back to an argument of
Ambrosio, reported by
Gigli~\cite[Proposition~3.3 and Corollary~3.4]{GigliOnHoelder}.
If the optimal map is Lipschitz, the conjugate potential is strongly
convex, and the duality gap controls the squared distance
to the optimal graph. Li and Nochetto~\cite[Proposition~3.1]{LiNochettoQuantitative}
formulate this coercivity for transport plans and use it to prove
stability under perturbations of both marginals. Their
Theorem~3.5 treats arbitrary $\cP_2$ perturbations when the reference
optimal map is Lipschitz.

For compact targets, Berman~\cite{BermanConvergenceRates} proved stability
estimates with dimension-dependent exponents under upper and lower
bounds on a source density over a convex body. M\'erigot, Delalande and
Chazal~\cite[Theorem~3.1]{MDCQuantitative} proved the dimension-independent
exponent $2/15$ for uniform sources on convex bodies.
Delalande and M\'erigot~\cite[Theorem~4.2]{DelalandeMerigotQuantitative}
improved this to
$W_1^{1/6}$ for source densities bounded above and away from zero on a
convex body. They also treated unbounded targets under moment
assumptions~\cite[Corollary~4.4]{DelalandeMerigotQuantitative}.
Letrouit and M\'erigot~\cite{LetrouitMerigotGluing}
used gluing arguments to treat more general sources, including
log-concave densities with full support in $\R^d$.
M\'erigot~\cite{MerigotSharpStability} proved the sharp $W_1^{1/4}$
estimate for bounded, compactly supported source densities, using
quantitative regularity estimates for convex functions due to Carlier,
Delalande and M\'erigot~\cite{CDMQuantitative}.
Ford~\cite{FordQuantitative} extended this approach to upper Ahlfors
regular sources
and to stability of plans under changes in both marginals.

The OpenAI preprint~\cite[Theorem~1.1]{OpenAISharpOneThird}
proves the sharp $W_2^{1/3}$ estimate for uniform
sources on compact convex bodies in dimensions $d\geq2$, with both
targets in a fixed compact set.

In this paper, we consider sources given by $L^p$ changes of BV reference
densities, with possibly unbounded support. We prove optimal H\"older
exponents for $W_2$ stability under a support or higher-moment bound on
one target, with explicit constants. The examples proving optimality use a fixed source
with unbounded, disconnected support for each $p$.

We regularize the Brenier potential to obtain a Lipschitz gradient and
apply the coercivity argument above. It then remains to estimate the
regularization error. We use the Moreau--Yosida envelope
\cite{MoreauProximite,BauschkeCombettesConvex}, whose conjugacy identity
$\phi_t^*=\phi^*+t\lvert\cdot\rvert^2\!/2$ makes the first-order changes of
the two potentials cancel after integration against the optimal plan.
A BV integration-by-parts estimate bounds the error by $O(t^2)$.
Interpolation extends this estimate to $L^p$ changes of density.
Quantitative Fenchel--Young inequalities involving proximal maps have
also been used in optimal transport by Carlier~\cite{CarlierFenchelYoung}.

The constants depend on the dimension through the source parameter
$\cS_p$ in \eqref{eq:parameters}. For log-concave sources, this dependence
is polynomial. For an $L^p$ change of Gaussian measure with a fixed
bound on the weight, the factor in the bounded-target estimate is
$d^{\alpha_p}$.

In both estimates, the second target ranges over all of $\cP_2(\R^d)$.
For bounded, compactly supported source densities, a one-sided support
condition follows from
M\'erigot~\cite[Theorem~3.4]{MerigotSharpStability} and the coupling argument in
Lemma~\ref{lem:gap-distance}; see Remark~\ref{rem:one-sided}.
Our moment estimate assumes a bound of any order $q>2$ on one target.
The moment estimates of
Delalande--M\'erigot~\cite[Theorem~4.3 and
Corollary~4.4]{DelalandeMerigotQuantitative}
assume higher-moment bounds on both targets.

\paragraph{AI use}
The use of potential regularization to prove stability arose in a
dialogue between the authors and OpenAI Codex (GPT-6). The AI first
proved a $W_2^{1/4}$ stability estimate for compactly supported targets
using heat regularization. The authors then asked it to determine
whether this exponent was optimal, extend the result to $L^p$ source
densities and noncompact targets, and compare regularization schemes.
This comparison led to the Moreau--Yosida argument, which improves the
dimension dependence for $p<\infty$. The AI wrote the entire manuscript,
and the authors checked it.

\paragraph{Funding}
This work benefited from support from the French State,
managed by the French National Research Agency (ANR)
as part of the France 2030 programme, reference ANR-23-IACL-0008.

\subsection{Main results}

Fix $p\in(1,\infty]$. We assume that the source and reference measures
satisfy
\begin{equation}\label{eq:source}
\begin{alignedat}{2}
\rho_0&=f_0\dd x\in\cP_2(\R^d),&\qquad f_0&\in\BV(\R^d),\\
\rho&=h\rho_0\in\cP_2(\R^d),&\qquad h&\in L^p(\rho_0).
\end{alignedat}
\end{equation}
Here $h$ is nonnegative and normalized by $\int h\dd\rho_0=1$.

Write
\[
\Lambda=\norm{h}_{L^p(\rho_0)},\qquad B_0=|Df_0|(\R^d).
\]
The variation $B_0$ includes boundary contributions when $f_0$ jumps
across the boundary of its support. For example,
$f_0=|K|^{-1}\one_K$ gives $B_0=\operatorname{Per}(K)/|K|$.
Let $m_\rho$ be the source mean and $M_\rho$ the root-mean-square
distance from this mean:
\[
m_\rho=\int x\dd\rho(x),\qquad
M_\rho=\biggl(\int|x-m_\rho|^2\dd\rho(x)\biggr)^{\!1/2}.
\]
Let $p'=p/(p-1)$, with $p'=1$ when $p=\infty$.
Define
\begin{equation}\label{eq:parameters}
\cS_p=M_\rho B_0\Lambda^{p'},\qquad
\alpha_p=\frac1{2(p'+1)}.
\end{equation}
The quantity $\cS_p$ is invariant under a common translation or scalar
dilation of the source and reference measures. If $\rho$ itself has a
BV density, we may take $\rho_0=\rho$, $h=1$ and $p=\infty$;
then $\cS_\infty=M_\rho B_0$ and $\alpha_\infty=1/4$.

\begin{theorem}[One bounded target]\label{thm:bounded}
Assume \eqref{eq:source}. Let $\mu$ be supported in a ball of radius
$R>0$, and let $\nu\in\cP_2(\R^d)$ be arbitrary. Then, with
$\delta=W_2(\mu,\nu)$,
\begin{equation}\label{eq:bounded}
\norm{T_\mu-T_\nu}_{L^2(\rho)}
\leq \delta+2\sqrt2\,R^{1-\alpha_p}(\cS_p\delta)^{\alpha_p}.
\end{equation}
In particular, $\alpha_p=(p-1)/(4p-2)$ and $\alpha_\infty=1/4$.
\end{theorem}

\begin{theorem}[One target with a bounded moment]\label{thm:moment}
Assume \eqref{eq:source}. Let $\nu\in\cP_2(\R^d)$ be arbitrary, and
suppose that, for some $q>2$ and $A>0$,
\[
\int |y|^q\dd\mu(y)\leq A^q.
\]
Set
\begin{equation}\label{eq:beta}
\beta_{p,q}=\frac{\alpha_p(q-2)}{q-2\alpha_p}
=\frac{q-2}{2q(p'+1)-2}.
\end{equation}
Then, with $\delta=W_2(\mu,\nu)$,
\begin{equation}\label{eq:moment}
\norm{T_\mu-T_\nu}_{L^2(\rho)}
\leq\delta+10A^{1-\beta_{p,q}}(\cS_p\delta)^{\beta_{p,q}}.
\end{equation}
For $p=\infty$, the exponent is $(q-2)/(4q-2)$.
\end{theorem}

The constants are uniform over $\nu\in\cP_2(\R^d)$, including measures
with infinite $q$-th moment in Theorem~\ref{thm:moment}.
The linear term $\delta$ accounts for large target distances.
By symmetry, either target may play the role of $\mu$.

\begin{remark}[Compact sources]\label{rem:W1}
Under \eqref{eq:source}, if $\supp\rho\subset\overline B(x_0,L)$ and
both targets lie in a common ball of radius $R$, then, with
$w=W_1(\mu,\nu)$,
\begin{equation}\label{eq:W1}
\norm{T_\mu-T_\nu}_{L^2(\rho)}
\leq(2Rw)^{1/2}+2\sqrt2\,R^{1-\alpha_p}
\biggl(\frac{L}{M_\rho}\cS_p w\biggr)^{\!\alpha_p}.
\end{equation}
See Remark~\ref{rem:W1-proof} for the derivation. For bounded, compactly
supported source densities, this recovers the known $W_1^{1/4}$ estimate,
whose sharpness in dimensions $d\geq2$ is due to
Letrouit~\cite[Theorem~2]{LetrouitUnstable}; see
Ford~\cite[Proposition~5.2]{FordQuantitative} for an extension of that construction.
\end{remark}

\begin{proposition}[Optimality of the exponents]\label{prop:sharpness}
For every $p\in(1,\infty]$, there is a fixed source on $\R^2$ satisfying
\eqref{eq:source} for which the following hold.
\begin{enumerate}[(i)]
\item There are targets $\mu_n$, $\nu_n$ in the unit ball with
$\delta_n=W_2(\mu_n,\nu_n)\to0$ such that, for every $a>\alpha_p$,
\[
\frac{\norm{T_{\mu_n}-T_{\nu_n}}_{L^2(\rho)}}{\delta_n^a}\longrightarrow\infty.
\]
\item For each $q>2$, there are targets with both $q$-moments at most one
and $\delta_n=W_2(\mu_n,\nu_n)\to0$ such that, for every $b>\beta_{p,q}$,
\[
\frac{\norm{T_{\mu_n}-T_{\nu_n}}_{L^2(\rho)}}{\delta_n^b}\longrightarrow\infty.
\]
\end{enumerate}
For $p=\infty$, the source itself has a bounded BV density.
\end{proposition}

We prove Proposition~\ref{prop:sharpness} in Section~\ref{sec:sharpness}.
Thus $1/4$ is the optimal $W_2$ exponent over the class of BV sources
with finite second moment and bounded targets.

\subsection{Examples and dimension dependence}

Take $h=1$ and $p=\infty$. The table lists upper bounds with a universal
constant $C$. The isotropic log-concave source has covariance $I_d$.
\begin{center}
\begin{tabular}{@{}lccc@{}}
\toprule
Source & $M_\rho$ & $B_0$ & $\cS_\infty = M_\rho B_0$\\
\midrule
$\gamma_d$ & $\sqrt d$ & $\sqrt d$ & $d$\\
Isotropic log-concave & $\sqrt d$ & $Cd$ & $Cd^{3/2}$\\
$Z^{-1}e^{-V}\dd x$, $\kappa I\leq D^2V\leq KI$
 & $\sqrt{d/\kappa}$ & $\sqrt{Kd}$ & $d\sqrt{K/\kappa}$\\
\bottomrule
\end{tabular}
\end{center}
For the Gaussian, $B_0=\E[|G_d|]$ with $G_d\sim N(0,I_d)$. The bound $B_0\leq Cd$ for
isotropic log-concave densities goes back to Eldan and
Klartag~\cite[Lemma~11]{EldanKlartagPointwise}, who proved
$\int|\nabla f_0|\dd x\leq Cd$ for smooth densities.
Gaussian smoothing, rescaling to isotropic position and lower
semicontinuity extend the bound to BV densities, including boundary
contributions. For a different proof and a
surface-area formulation, see Eskenazis, Giannopoulos and
Tziotziou~\cite[Theorems~1.3 and~1.4]{EGTFunctionalPerimeter}.
In the last line, $V\in C^2(\R^d)$
and $0<\kappa\leq K$. By the Poincar\'e inequality,
$M_\rho^2\leq d/\kappa$, and by integration by parts,
$B_0^2\leq\int|\nabla V|^2\dd\rho=\int\Delta V\dd\rho\leq Kd$.

For targets in a ball of radius $R$, Theorem~\ref{thm:bounded} becomes
\begin{equation}\label{eq:examples}
\norm{T_\mu-T_\nu}_2\leq CR\biggl(\frac{\delta}{R}\biggr)^{\!1/4}
\begin{cases}
d^{1/4},&\rho=\gamma_d,\\
d^{3/8},&\rho\text{ isotropic log-concave},\\
d^{1/4}(K/\kappa)^{1/8},&\kappa I\leq D^2V\leq KI.
\end{cases}
\end{equation}
We absorbed the linear term using $\delta\leq2R$.
Under the common moment bound of Theorem~\ref{thm:moment}, the factor
$CR(\delta/R)^{1/4}$ becomes $CA(\delta/A)^{\beta_{\infty,q}}$.
The three dimension factors become
$d^{\beta_{\infty,q}}$, $d^{3\beta_{\infty,q}/2}$ and
$d^{\beta_{\infty,q}}(K/\kappa)^{\beta_{\infty,q}/2}$, respectively.

For $\rho=h\gamma_d$ with $\norm h_{L^p(\gamma_d)}=\Lambda$, H\"older's
inequality and Gaussian moment bounds imply
\[
M_\rho\leq C_p\Lambda^{1/2}\sqrt d,\qquad
\cS_p\leq C_p\Lambda^{p'+1/2}d.
\]
Thus, for fixed $p>1$ and both targets in a common ball of radius $R$,
\begin{equation}\label{eq:weightedGaussian}
\norm{T_\mu-T_\nu}_2\leq
C_p\Lambda^{\frac{2p'+1}{4(p'+1)}}d^{\alpha_p}
R^{1-\alpha_p}\delta^{\alpha_p}.
\end{equation}
Under a common $q$-moment bound on the targets, the same computation gives
\[
\norm{T_\mu-T_\nu}_2\leq
C_{p,q}\Lambda^{(p'+1/2)\beta_{p,q}}d^{\beta_{p,q}}
A^{1-\beta_{p,q}}\delta^{\beta_{p,q}}.
\]
Here $\delta\leq2A$ absorbs the linear term. For the uniform source on
$B(0,L)$, $LB_0=d$ and \eqref{eq:W1} gives
$\norm{T_\mu-T_\nu}_2\leq C R^{3/4}d^{1/4}W_1(\mu,\nu)^{1/4}$.

For Gaussian sources, the next construction gives a polynomial lower
bound on the dimension dependence in Theorems~\ref{thm:bounded}
and~\ref{thm:moment}.

\begin{proposition}[Dimension dependence for Gaussian sources]\label{prop:dimension}
There are universal $c,C>0$ such that, in every sufficiently large
dimension $d$, the source $\rho=\gamma_d=N(0,I_d)$ admits targets
$\mu_d$, $\nu_d$ in the unit ball satisfying
\[
W_2(\mu_d,\nu_d)\leq C d^{-1/2},\qquad
\norm{T_{\mu_d}-T_{\nu_d}}_{L^2(\gamma_d)}\geq c.
\]
Consequently, a uniform estimate $\norm{T_\mu-T_\nu}_2\leq C_dW_2(\mu,\nu)^a$
with $a>0$ requires $C_d\geq c_a d^{a/2}$.
\end{proposition}

\begin{proof}
First let $d=2n$. Choose a binary linear code $\mathcal C\subset\{0,1\}^n$
of dimension $k=\lfloor n/4\rfloor$ whose nonzero words have weight
greater than $n/8$. Such codes exist for large $n$: for a random linear
map from $\{0,1\}^k$ to $\{0,1\}^n$, a union bound gives a failure
probability at most $2^{k-n(1-H_2(1/8))}$, which tends to zero; here
$H_2$ is the binary entropy and $1-H_2(1/8)>1/4$.

Write the associated sign vectors as
$v_i=n^{-1/2}((-1)^{c_j})_{j=1}^n$, $1\leq i\leq N=2^k$.
They have unit norm, satisfy $|v_i-v_j|^2\geq1/2$ for $i\ne j$,
and are permuted transitively by orthogonal coordinate sign changes.
Put
\[
r=\frac{\log N-1}{\log N+1},\qquad
a_i=(v_i,0),\qquad b_i=(rv_i,\sqrt{1-r^2}\,v_i),
\]
and let $\mu_d$, $\nu_d$ be uniform on the $a_i$, $b_i$, respectively.
For independent standard Gaussian vectors $X,Z\in\R^n$, set
$Y=rX+\sqrt{1-r^2}Z$ and $I(x)=\arg\max_i\ip{v_i}{x}$.
Ties have probability zero, and transitivity implies
$\PP[I(X)=i]=1/N$. The gradients of the two maxima of linear functions
are therefore the Brenier maps:
\[
T_{\mu_d}(X,Z)=a_{I(X)},\qquad T_{\nu_d}(X,Z)=b_{I(Y)}.
\]
Nelson's Gaussian hypercontractivity theorem
(see~\cite[Theorem~2.1]{ChafaiLehecLSI}) implies, for the cells
$A_i=\{I=i\}$,
\[
\PP[I(X)=I(Y)]
\leq\sum_i\gamma_n(A_i)^{2/(1+r)}
=N^{-(1-r)/(1+r)}=e^{-1}.
\]
Since $|a_i-b_j|^2=2(1-r)+r|v_i-v_j|^2$, diagonal matching is optimal
between the two uniform targets. Therefore
\[
W_2(\mu_d,\nu_d)^2=\frac4{\log N+1},\qquad
\norm{T_{\mu_d}-T_{\nu_d}}_2^2\geq\frac r2(1-e^{-1}).
\]
For odd dimensions, add an independent Gaussian coordinate on which
both maps are constant.
\end{proof}

These targets have moments of every order bounded by one, so the lower
bound applies to both main theorems. For exponent $1/4$, the constant is
bounded below by $cd^{1/8}$ and above by $Cd^{1/4}$. For the moment
estimate with exponent $\beta_{\infty,q}$, the bounds are
$cd^{\beta_{\infty,q}/2}$ and $Cd^{\beta_{\infty,q}}$, respectively.
The optimal power of the dimension between these bounds is open.

\section{One bounded target}\label{sec:bounded}

For a proper lower semicontinuous convex function $u$, write
$u^*(y)=\sup_x\{\ip{x}{y}-u(x)\}$ and
$g_u(x,y)=u(x)+u^*(y)-\ip{x}{y}$ for its Fenchel gap.
Let $\phi$ be a Brenier potential from $\rho$ to $\mu$, and set
$\sigma_*=(\Id,T_\mu)_\#\rho$. For $\sigma\in\Pi(\rho,\mu)$, define
\begin{equation}\label{eq:gap}
\begin{aligned}
\cG(\sigma)&=\int g_\phi\dd\sigma\\
&=\int\ip{x}{T_\mu(x)}\dd\rho(x)
-\int\ip{x}{y}\dd\sigma(x,y)\geq0.
\end{aligned}
\end{equation}
The second equality follows from Fenchel equality on $\sigma_*$.
Thus $\cG(\sigma)$ is the difference between the optimal correlation
and the correlation under $\sigma$, and $\cG(\sigma_*)=0$.

When $\supp\mu\subset\overline B(0,R)$, we choose $\phi$ globally convex
and $R$-Lipschitz by taking the supremum of its supporting affine
functions with slopes in this ball. Define its Moreau--Yosida
regularization and the corresponding gap by
\begin{equation}\label{eq:regularizedgap}
\phi_t(x)=\inf_z{\biggl\{\phi(z)+\frac{|x-z|^2}{2t}\biggr\}},\qquad
\cG_t(\sigma)=\int g_{\phi_t}\dd\sigma,\quad t>0.
\end{equation}
The envelope is convex and continuously differentiable, and
\begin{equation}\label{eq:Moreau-dual}
\phi_t^*(y)=\phi^*(y)+\frac t2|y|^2.
\end{equation}
For the properties of the envelope, see
Moreau~\cite{MoreauProximite} and
Bauschke--Combettes~\cite[Chapter~12]{BauschkeCombettesConvex}.
The proof has three steps: estimate the regularization error, use
coercivity of the regularized gap, and bound the original gap by the
target distance through a suitable coupling.

\begin{lemma}[Stability of the duality gap]\label{lem:gap-stability}
Assume \eqref{eq:source}, and set $H_p=\Lambda B_0^{1/p'}$.
For any convex $R$-Lipschitz function $\phi$, define almost everywhere
\begin{equation}\label{eq:pointwise-bias}
a_t=\phi_t-\phi+\frac t2|\nabla\phi|^2,\qquad
A_t=\int a_t\dd\rho.
\end{equation}
Then
\begin{equation}\label{eq:bias}
A_t\leq\frac{H_p}{2\,3^{1/p'}}
R^{2+1/p'}t^{1+1/p'}.
\end{equation}
If $\nabla\phi$ pushes $\rho$ to $\mu$, then for every
$\sigma\in\Pi(\rho,\mu)$,
\begin{equation}\label{eq:gap-stability}
0\leq\cG_t(\sigma)-\cG(\sigma)
=\cG_t(\sigma_*)=A_t.
\end{equation}
\end{lemma}

\begin{proof}
By \eqref{eq:Moreau-dual} and Fenchel equality,
$a_t(x)=g_{\phi_t}(x,\nabla\phi(x))\geq0$ almost everywhere.
Since $\phi_t\leq\phi$ and $|\nabla\phi|\leq R$, we also have
$0\leq a_t\leq tR^2/2$. If $(\nabla\phi)_\#\rho=\mu$, the
common marginals and \eqref{eq:Moreau-dual} give
\[
\cG_t(\sigma)-\cG(\sigma)
=\int(\phi_t-\phi)\dd\rho+\frac t2\int|y|^2\dd\mu=A_t.
\]
Evaluating at $\sigma_*$ proves \eqref{eq:gap-stability}.
All integrals are finite: the potentials have linear growth and
$\phi^*$ is integrable against $\mu$ by Fenchel equality on $\sigma_*$.

We next prove the second-order estimate
\begin{equation}\label{eq:BV-bias}
\int a_t\dd\rho_0\leq\frac{R^3B_0}{6}t^2.
\end{equation}
First take $\rho=\rho_0$ and suppose that $\phi$ is smooth.
The Moreau envelope satisfies
the Hamilton--Jacobi equation
\[
\partial_s\phi_s+\frac12|\nabla\phi_s|^2=0,\qquad \phi_0=\phi.
\]
Differentiating in time, we find $A_0=A'_0=0$ and
\[
A''_s=\int\ip{\nabla\phi_s}{D^2\phi_s\nabla\phi_s}\,f_0\dd x.
\]
Since $D^2\phi_s\geq0$, we have
\[
\operatorname{div}(|\nabla\phi_s|^2\nabla\phi_s)
\geq3\ip{\nabla\phi_s}{D^2\phi_s\nabla\phi_s}.
\]
Integration by parts and the bound $|\nabla\phi_s|\leq R$ imply
\begin{equation}\label{eq:bias-second-derivative}
0\leq A''_s
\leq-\frac13\int |\nabla\phi_s|^2\nabla\phi_s\cdot\dd Df_0
\leq\frac{R^3B_0}{3}.
\end{equation}
Integrating twice in time, we obtain
\[
A_t=\int_0^t(t-s)A''_s\dd s
\leq\frac{R^3B_0}{6}t^2.
\]
In the integration by parts, cutoffs on balls of radius $a$
produce an error bounded by $CR^3/a$, which tends to zero.
For a general convex $R$-Lipschitz $\phi$, apply the estimate to smooth
convolutions. As the mollification parameter vanishes, the potentials
and their envelopes converge uniformly, while the gradients of the initial
potentials converge almost everywhere to $\nabla\phi$ and remain
bounded by $R$. Dominated convergence proves \eqref{eq:BV-bias} for
the original convex Lipschitz potential.

For $\rho=h\rho_0$ and $1<p<\infty$, H\"older's inequality gives
\[
A_t=\int a_t h\dd\rho_0
\leq\Lambda\biggl(\int a_t^{p'}\dd\rho_0\biggr)^{\!1/p'}.
\]
Using $0\leq a_t\leq tR^2/2$ and \eqref{eq:BV-bias}, we have
\[
\int a_t^{p'}\dd\rho_0
\leq\biggl(\frac{tR^2}{2}\biggr)^{\!p'-1}\int a_t\dd\rho_0
\leq\biggl(\frac{tR^2}{2}\biggr)^{\!p'-1}\frac{R^3B_0t^2}{6}.
\]
Taking the $1/p'$ power and using $(p'-1)/p'=1/p$, we conclude that
\[
A_t\leq\Lambda
\biggl(\frac{tR^2}{2}\biggr)^{\!1/p}
\biggl(\frac{R^3B_0t^2}{6}\biggr)^{\!1/p'}
=\frac{H_p}{2\,3^{1/p'}}R^{2+1/p'}t^{1+1/p'}.
\]
For $p=\infty$, the same bound follows directly from
$A_t\leq\Lambda\int a_t\dd\rho_0$ and \eqref{eq:BV-bias}, with $p'=1$.
\end{proof}

\begin{lemma}[Coercivity of the regularized gap]\label{lem:coercivity}
For $t>0$, the regularized potential satisfies
\begin{equation}\label{eq:smooth}
\norm{\nabla\phi_t}_\infty\leq R,\qquad
\Lip(\nabla\phi_t)\leq\frac1t=:L_t.
\end{equation}
For every $x,y\in\R^d$,
\begin{equation}\label{eq:Fenchel}
g_{\phi_t}(x,y)\geq\frac t2|y-\nabla\phi_t(x)|^2.
\end{equation}
Integrating against any $\sigma\in\Pi(\rho,\mu)$, we obtain
\begin{equation}\label{eq:coercivity}
\int|y-\nabla\phi_t(x)|^2\dd\sigma(x,y)\leq\frac2t\cG_t(\sigma).
\end{equation}
\end{lemma}

\begin{proof}
Let $z_t(x)$ be the unique minimizer in \eqref{eq:regularizedgap}.
The optimality condition gives
\[
\nabla\phi_t(x)=\frac{x-z_t(x)}t\in\partial\phi(z_t(x)),
\]
so its norm is at most $R$. Monotonicity of $\partial\phi$ implies
\[
t|\nabla\phi_t(x)-\nabla\phi_t(x')|^2
\leq\ip{\nabla\phi_t(x)-\nabla\phi_t(x')}{x-x'},
\]
which proves the Lipschitz bound. In the definition of $\phi_t^*(y)$,
use the inequality
\[
\phi_t(z)\leq\phi_t(x)+\ip{\nabla\phi_t(x)}{z-x}
+\frac1{2t}|z-x|^2
\]
at $z=x+t(y-\nabla\phi_t(x))$ to obtain \eqref{eq:Fenchel}.
Integration proves \eqref{eq:coercivity}. This is the
smoothness--strong-convexity argument used by Li and
Nochetto~\cite[Lemma~2.2 and Proposition~3.1]{LiNochettoQuantitative}.
For proximal versions of the Fenchel--Young inequality with a remainder,
see Carlier~\cite{CarlierFenchelYoung}.
\end{proof}

\begin{lemma}[The target distance controls the gap]\label{lem:gap-distance}
Let $\mu,\nu\in\cP_2(\R^d)$ and $\delta=W_2(\mu,\nu)$.
There exist jointly distributed $(X_\nu,Y_\nu,Y_\mu)$ with
$X_\nu\sim\rho$, $Y_\nu=T_\nu(X_\nu)$ and $(Y_\mu,Y_\nu)$ an optimal
coupling of the targets, such that $\sigma=\operatorname{Law}(X_\nu,Y_\mu)$
satisfies
\begin{equation}\label{eq:gap-distance}
\cG(\sigma)\leq2M_\rho\delta.
\end{equation}
Here $\cG$ is defined by \eqref{eq:gap}.
\end{lemma}

\begin{proof}
Glue a $W_2$-optimal target coupling to the two optimal graph plans to
obtain $(X_\mu,X_\nu,Y_\mu,Y_\nu)$ with $X_\eta\sim\rho$ and
$Y_\eta=T_\eta(X_\eta)$ for $\eta\in\{\mu,\nu\}$.
Optimality of $T_\nu$ gives
$\E[\ip{X_\mu}{Y_\nu}]\leq\E[\ip{X_\nu}{Y_\nu}]$. Hence
\begin{align*}
\cG(\sigma)
&=\E[\ip{X_\mu-X_\nu}{Y_\mu}]\\
&\leq\E[\ip{X_\mu-X_\nu}{Y_\mu-Y_\nu}]\\
&\leq\norm{X_\mu-X_\nu}_2\norm{Y_\mu-Y_\nu}_2
\leq2M_\rho\delta.
\end{align*}
The last bound uses
$\norm{X_\mu-X_\nu}_2\leq\norm{X_\mu-m_\rho}_2+
\norm{X_\nu-m_\rho}_2$.
\end{proof}

\begin{remark}[A one-sided support condition]\label{rem:one-sided}
For bounded, compactly supported source densities and bounded $\mu$,
M\'erigot's coercivity estimate~\cite[Theorem~3.4]{MerigotSharpStability}
combined with Lemma~\ref{lem:gap-distance} gives
$\norm{T_\mu-T_\nu}_{L^2(\rho)}\leq\delta+C\delta^{1/4}$ for arbitrary
$\nu\in\cP_2(\R^d)$, where $\delta=W_2(\mu,\nu)$.
\end{remark}

\begin{proof}[Proof of Theorem~\ref{thm:bounded}]
Translate both targets to arrange $\supp\mu\subset\overline B(0,R)$.
Choose the coupling $\sigma$ from Lemma~\ref{lem:gap-distance}.
Applying Lemma~\ref{lem:coercivity} to $\sigma$ and $\sigma_*$ gives
\begin{align}
\int|T_\mu(x)-y|^2\dd\sigma(x,y)
&\leq4L_t\bigl(\cG_t(\sigma_*)+\cG_t(\sigma)\bigr)\notag\\
&\leq4L_t\cG(\sigma)+8L_tA_t,\label{eq:graph}
\end{align}
where the second inequality is Lemma~\ref{lem:gap-stability}.
Since $Y_\nu=T_\nu(X_\nu)$ and
$\norm{Y_\mu-Y_\nu}_2=\delta$, the $L^2$ triangle inequality implies,
with $D=\norm{T_\mu-T_\nu}_{L^2(\rho)}$,
\begin{equation}\label{eq:balance}
(D-\delta)_+^2\leq\frac{8M_\rho\delta}{t}
+\frac4{3^{1/p'}}H_pR^{2+1/p'}t^{1/p'}.
\end{equation}
For $\delta>0$, choose
\[
s=\biggl(\frac{M_\rho\delta}{H_pR^{2+1/p'}}\biggr)^{\!p'/(p'+1)},
\qquad t=2s.
\]
Since $(2/3)^{1/p'}\leq1$, the right-hand side of
\eqref{eq:balance} is at most $8M_\rho\delta/s$.
Using $M_\rho H_p^{p'}=\cS_p$ and taking square roots gives
\eqref{eq:bounded}. When $\delta=0$, the two maps agree by uniqueness.
\end{proof}

\begin{remark}[$W_1$ stability for compact sources]\label{rem:W1-proof}
To prove \eqref{eq:W1}, write $w=W_1(\mu,\nu)$ and use the gluing
construction of Lemma~\ref{lem:gap-distance} with a $W_1$-optimal target
coupling. Since $|X_\mu-X_\nu|\leq2L$, the same gap estimate gives
\[
\cG(\sigma)\leq2L\E[|Y_\mu-Y_\nu|]=2Lw.
\]
The common target ball also gives
$\norm{Y_\mu-Y_\nu}_2^2\leq2R\E[|Y_\mu-Y_\nu|]=2Rw$.
Thus the proof of Theorem~\ref{thm:bounded} applies with
$M_\rho\delta$ replaced by $Lw$ in \eqref{eq:balance} and the initial
term $\delta$ replaced by $(2Rw)^{1/2}$. Balancing the two terms in $t$
gives \eqref{eq:W1}.

Every bounded probability density $f$ supported in $B(x_0,L)$ satisfies
\eqref{eq:source} with $p=\infty$: take $\rho_0$ uniform on this ball
and $h=|B(x_0,L)|f$. Thus \eqref{eq:W1} applies to every such density
with exponent $1/4$.
\end{remark}

\begin{remark}[Heat regularization]\label{rem:heat}
The proof also applies to the heat regularization
$P_t\phi$, where $P_t=e^{t\Delta}$. Convexity gives $P_t\phi\geq\phi$
and $(P_t\phi)^*\leq\phi^*$, so the change in the duality gap is
bounded by $\int(P_t\phi-\phi)\dd\rho$.
Set $\kappa_d=\sqrt2\,\E[|G_d|]$ for $G_d\sim N(0,I_d)$.
The pointwise bound $0\leq P_t\phi-\phi\leq R\kappa_d\sqrt t$
and the BV estimate
$\int(P_t\phi-\phi)\dd\rho_0\leq RB_0t$, combined with
H\"older's inequality, imply
\[
\int(P_t\phi-\phi)\dd\rho
\leq R\Lambda\kappa_d^{1/p}B_0^{1/p'}t^{(1+1/p')/2}.
\]
The regularized gradient satisfies
\[
\Lip(\nabla P_t\phi)\leq\frac{R}{\sqrt{\pi t}}.
\]
The same coercivity and balancing argument proves a bound of the form
\eqref{eq:bounded}, with a universal numerical constant and source
parameter $\cS_p\kappa_d^{p'-1}$. The H\"older exponent is $\alpha_p$.
Since $\kappa_d$ is of order $\sqrt d$, the extra factor contributes to
the dimension dependence when $1<p<\infty$ and equals one when
$p=\infty$.
\end{remark}

\section{One target with a bounded moment}\label{sec:unbounded}

For $\mu\in\cP_2(\R^d)$, define
\[
\tau_\mu(R)=\biggl(\int_{|y|>R}|y|^2\dd\mu(y)\biggr)^{\!1/2}.
\]
We extend the bounded-target estimate by adding an error term involving
the tail of $\mu$.

\begin{proposition}[Stability with a target-tail error]\label{prop:tailestimate}
Assume \eqref{eq:source}, and let $\mu$, $\nu\in\cP_2(\R^d)$.
For every $R>0$ such that $\mu(\overline B(0,R))>0$, with
$\delta=W_2(\mu,\nu)$, one has
\begin{equation}\label{eq:tailestimate}
\norm{T_\mu-T_\nu}_{L^2(\rho)}
\leq\delta+2\sqrt2\,R^{1-\alpha_p}(\cS_p\delta)^{\alpha_p}
+3\sqrt2\,\tau_\mu(R).
\end{equation}
\end{proposition}

\begin{proof}
The case $\delta=0$ follows from uniqueness, so assume $\delta>0$.
Let $\phi$ be a proper lower semicontinuous Brenier potential from
$\rho$ to $\mu$. Define the slope truncation
\begin{equation}\label{eq:truncation}
\phi_R(x)=(\phi^*+\iota_{\overline B(0,R)})^*(x)
=\sup_{|y|\leq R}{\{\ip{x}{y}-\phi^*(y)\}},
\end{equation}
where $\iota_E$ is zero on $E$ and $+\infty$ outside $E$.
Since $\phi^*$ is finite $\mu$-almost everywhere and
$\mu(\overline B(0,R))>0$, the supremum includes a finite affine function.
Lower semicontinuity on the compact ball bounds $\phi^*$ from below
there. Hence $\phi_R$ is finite, convex and $R$-Lipschitz, with
\begin{equation}\label{eq:truncateddual}
\phi_R\leq\phi,\qquad
\phi_R^*=\phi^*+\iota_{\overline B(0,R)}.
\end{equation}
This is the Pasch--Hausdorff Lipschitz regularization,
\[
\phi_R(x)=\inf_z{\{\phi(z)+R|x-z|\}};
\]
see
Bauschke and Combettes~\cite[Proposition~12.17]{BauschkeCombettesConvex}
for the envelope properties and Beck and
Teboulle~\cite[Lemma~2.1]{BeckTeboulleExact} for its dual representation.

Write $T=\nabla\phi$, $T_R=\nabla\phi_R$ and let
$\phi_{R,t}$ be the Moreau--Yosida regularization of $\phi_R$.
Set
\[
a_{R,t}=\phi_{R,t}-\phi_R+\frac t2|T_R|^2,\qquad
A_{R,t}=\int a_{R,t}\dd\rho.
\]
Applying Lemma~\ref{lem:gap-stability} to $\phi_R$, we obtain
\begin{equation}\label{eq:truncated-bias}
A_{R,t}\leq\frac{H_p}{2\,3^{1/p'}}R^{2+1/p'}t^{1+1/p'}.
\end{equation}
Also $|T_R|$, $|\nabla\phi_{R,t}|\leq R$. Wherever $|T(x)|\leq R$,
Fenchel equality gives $\phi_R(x)=\phi(x)$, hence $T_R(x)=T(x)$
almost everywhere on that set. Consequently,
\begin{equation}\label{eq:truncated-moment}
0\leq\int\bigl(|T|^2-|T_R|^2\bigr)\dd\rho
\leq\tau_\mu(R)^2.
\end{equation}

For $|y|\leq R$, the conjugacy identities
\eqref{eq:truncateddual} and \eqref{eq:Moreau-dual} imply
\begin{equation}\label{eq:smallslopes}
g_{\phi_{R,t}}(x,y)
\leq g_\phi(x,y)+a_{R,t}(x)
+\frac t2\bigl(|y|^2-|T_R(x)|^2\bigr).
\end{equation}
For any $\sigma\in\Pi(\rho,\mu)$, the function
$|y|^2-|T_R(x)|^2$ is nonnegative on $\{|y|>R\}$.
Thus its integral over $\{|y|\leq R\}$ is bounded
above by its integral over all pairs, which is at most
$\tau_\mu(R)^2$ by \eqref{eq:truncated-moment}.
Since $g_\phi$ and $a_{R,t}$ are nonnegative, coercivity gives
\[
\int_{|y|\leq R}|\nabla\phi_{R,t}(x)-y|^2\dd\sigma
\leq\frac2t\bigl(\cG(\sigma)+A_{R,t}\bigr)+\tau_\mu(R)^2.
\]
On $\{|y|>R\}$, use
$|\nabla\phi_{R,t}(x)-y|^2\leq4|y|^2$ instead. We obtain
\begin{equation}\label{eq:tailgraph1}
\int|\nabla\phi_{R,t}(x)-y|^2\dd\sigma
\leq\frac2t\bigl(\cG(\sigma)+A_{R,t}\bigr)+5\tau_\mu(R)^2.
\end{equation}
Here $\cG$ is the gap for the original potential $\phi$ as in
\eqref{eq:gap}. Its potential terms are integrable: affine lower
bounds for $\phi$ and $\phi^*$ control their negative parts, while
Fenchel equality on the optimal graph and the second moments control
their positive parts.

On the optimal graph, $T_R=T$ when $|T|\leq R$, so the moment correction
vanishes on this set:
\[
\int|\nabla\phi_{R,t}-T|^2\dd\rho
\leq\frac2t A_{R,t}+4\tau_\mu(R)^2.
\]
Combining this with \eqref{eq:tailgraph1} gives
\[
\int|T(x)-y|^2\dd\sigma
\leq\frac4t\cG(\sigma)+\frac8t A_{R,t}+18\tau_\mu(R)^2.
\]
Choose the coupling from Lemma~\ref{lem:gap-distance} and use
\eqref{eq:truncated-bias}. We bound the first two terms as
in \eqref{eq:balance}. The same choice of $t$, followed by
$\sqrt{a+b}\leq\sqrt a+\sqrt b$ and the target-coupling triangle
inequality, proves \eqref{eq:tailestimate}.
\end{proof}

\begin{remark}
The estimate \eqref{eq:smallslopes} bounds the regularized gap by the
original Fenchel gap, the regularization error and a moment correction.
Lemma~\ref{lem:gap-distance} therefore applies with the original targets
$\mu$ and $\nu$, while
\eqref{eq:truncated-moment} bounds the correction by the tail of $\mu$.
Delalande--M\'erigot~\cite[Theorem~4.3 and
Corollary~4.4]{DelalandeMerigotQuantitative} treat unbounded targets
using moment estimates and local control of potentials.
\end{remark}

\begin{proof}[Proof of Theorem~\ref{thm:moment}]
Write $\alpha=\alpha_p$, $\beta=\beta_{p,q}$ and
$z=\cS_p\delta/A$. The moment bound gives
\begin{equation}\label{eq:tailmoment}
\tau_\mu(R)\leq A^{q/2}R^{-(q-2)/2}.
\end{equation}
If $0<z\leq1$, choose
\[
R=2A z^{-1/[q(p'+1)-1]}.
\]
Then $R\geq2A$, so $\mu(\overline B(0,R))>0$.
The two error terms in Proposition~\ref{prop:tailestimate} have the
same power $z^\beta$, because
\[
\alpha-\frac{1-\alpha}{q(p'+1)-1}
=\frac{q-2}{2[q(p'+1)-1]}=\beta.
\]
Their sum is bounded by
\[
\bigl(2\sqrt2\,2^{1-\alpha}+3\sqrt2\,2^{-(q-2)/2}\bigr)Az^\beta
\leq7\sqrt2\,Az^\beta<10Az^\beta.
\]
If $z\geq1$, a $W_2$-optimal target coupling and Minkowski's inequality give
\[
\norm{T_\nu}_{L^2(\rho)}
=\biggl(\int|y|^2\dd\nu(y)\biggr)^{\!1/2}\leq A+\delta.
\]
Hence $\norm{T_\mu-T_\nu}_2\leq2A+\delta\leq\delta+10Az^\beta$.
The case $\delta=0$ follows from uniqueness.
\end{proof}

\section{Optimality of the exponents}\label{sec:sharpness}

\begin{proof}[Proof of Proposition~\ref{prop:sharpness}]
Fix $p\in(1,\infty]$ and its conjugate $p'$. We construct one source
whose mass is concentrated on pairs of thin rectangles escaping to
infinity. For integers $n\geq N$, with $N$ sufficiently large, set
\begin{equation}\label{eq:sharp-scales}
\begin{gathered}
m_n=e^{-n^2},\qquad b_n=n^{2(p'-1)}m_n^{p'},\qquad
\ell_n=n^2b_n,\\
L_n=\frac1{n\sqrt{m_n}},\qquad c_n=1-2^{-n}.
\end{gathered}
\end{equation}
Let $u_n^\pm$ be the uniform probability measures on the rectangles
\[
Q_n^\pm=
\biggl[c_n\pm\ell_n-\frac{\ell_n}{4},
      c_n\pm\ell_n+\frac{\ell_n}{4}\biggr]
\times[\mp L_n-1,\mp L_n+1],
\]
and let $u_0$ be uniform on $[-2,-1]\times[-1,1]$. Define
\begin{align*}
\rho&=\biggl(1-\sum_{n\geq N}m_n\biggr)u_0
       +\sum_{n\geq N}\frac{m_n}{2}(u_n^++u_n^-),\\
\rho_0&=\biggl(1-\sum_{n\geq N}b_n\biggr)u_0
       +\sum_{n\geq N}\frac{b_n}{2}(u_n^++u_n^-).
\end{align*}
These are probabilities on pairwise disjoint rectangles. Since
$b_n\leq m_n$ and $m_nL_n^2=n^{-2}$, both measures have finite second
moments. Each $Q_n^\pm$ has area $\ell_n$ and perimeter $4+\ell_n$,
so the density $f_0$ of $\rho_0$ satisfies
\[
|Df_0|(\R^2)\leq C+\sum_{n\geq N}
       b_n\frac{4+\ell_n}{\ell_n}<\infty.
\]
For $p<\infty$, the density $h=\dd\rho/\dd\rho_0$ obeys
\[
\int h^p\dd\rho_0
=C_p+\sum_{n\geq N}\frac{m_n^p}{b_n^{p-1}}
=C_p+\sum_{n\geq N}\frac1{n^2}<\infty.
\]
For $p=\infty$, $b_n=m_n$, so $\rho=\rho_0$ and $h=1$.
The density is then bounded, since its value on $Q_n^\pm$ is
$m_n/(2\ell_n)=1/(2n^2)$.

Put $s_n=\sum_{j>n}m_j$ and $r_n=m_n/2+s_n$, so that
$r_n\sim m_n/2$ and $s_n<m_n/4$. The half-space $\{x_1>c_n\}$
contains $Q_n^+$ and all rectangles with index $j>n$, hence has
$\rho$-mass $r_n$. Set
\[
\eta_n=\frac{4\ell_n}{L_n}
       =4n^{2p'+1}m_n^{p'+1/2},\qquad
v_n=\frac{e_1+\eta_ne_2}{\sqrt{1+\eta_n^2}}.
\]
Under the projection $F_n(x)=x_1+\eta_nx_2$, the two rectangles
exchange order:
\[
F_n(Q_n^+)\subset
\biggl[c_n-\frac72\ell_n,c_n-\frac52\ell_n\biggr],\qquad
F_n(Q_n^-)\subset
\biggl[c_n+\frac52\ell_n,c_n+\frac72\ell_n\biggr].
\]
Their centers project to $c_n\mp3\ell_n$ and their projected
half-width is at most $\ell_n/4+\eta_n\leq\ell_n/2$.
Every rectangle with index $j<n$ projects below
$c_n-4\ell_n$. For $N$ large, the right edges of the rectangles
increase with their index and $L_{n-1}+1\leq L_n$, so their projections
are at most
\[
c_{n-1}+\frac54\ell_{n-1}+4\ell_n<c_n-4\ell_n.
\]
The last inequality follows from $\ell_n\leq n^2e^{-n^2}$ and
$c_n-c_{n-1}=2^{-n}$. The base rectangle satisfies the same bound
directly.

Choose $t_n$ so that $\rho(F_n>t_n)=r_n$; the projected distribution
has no atoms. The two current rectangles lie above all earlier ones
and have mass $m_n>r_n$. Since the later rectangles have total mass
$s_n$, the level set must take at least $r_n-s_n=m_n/2$ from the current
pair, starting with $Q_n^-$. It therefore contains all of $Q_n^-$ and
at most mass $s_n$ from $Q_n^+$.

For any amplitude $a_n>0$, consider the two-point targets
\[
\mu_n=(1-r_n)\delta_0+r_n\delta_{a_ne_1},\qquad
\nu_n=(1-r_n)\delta_0+r_n\delta_{a_nv_n}.
\]
Their Brenier maps are
\[
T_{\mu_n}=a_ne_1\one_{\{x_1>c_n\}},\qquad
T_{\nu_n}=a_nv_n\one_{\{F_n>t_n\}},
\]
as gradients of scaled positive parts of affine functions.
On mass at least $m_n-s_n$, exactly one of these maps vanishes.
Matching corresponding target atoms, we obtain
\begin{equation}\label{eq:sharp-rates}
\begin{aligned}
\norm{T_{\mu_n}-T_{\nu_n}}_2&\geq a_n\sqrt{m_n-s_n},\\
\delta_n&\leq a_n\sqrt{r_n}\,|e_1-v_n|
\leq C a_n n^{2p'+1}m_n^{p'+1}.
\end{aligned}
\end{equation}
Taking $a_n=1$ gives targets in the unit ball. Since $m_n=e^{-n^2}$,
\eqref{eq:sharp-rates} rules out every exponent larger than
$1/[2(p'+1)]=\alpha_p$.

For $q>2$, take $a_n=r_n^{-1/q}$. Both targets then
have $q$-moment exactly one, while
\[
\norm{T_{\mu_n}-T_{\nu_n}}_2\geq c m_n^{1/2-1/q},\qquad
\delta_n\leq Cn^{2p'+1}m_n^{p'+1-1/q}.
\]
Every exponent larger than
$(1/2-1/q)/(p'+1-1/q)=\beta_{p,q}$ is therefore impossible.
\end{proof}

\sloppy
\printbibliography
\end{document}